\documentclass[a4paper,12pt]{amsart}
\usepackage[utf8]{inputenc}
\usepackage{amsmath,amssymb,amsfonts,amsthm,graphicx}

\newtheorem{theorem}{Theorem}[section]
\newtheorem{lemma}[theorem]{Lemma}
\newtheorem{proposition}[theorem]{Proposition}
\newtheorem{corollary}[theorem]{Corollary}
\numberwithin{equation}{section}

\theoremstyle{remark}
\newtheorem{remark}[theorem]{Remark}
\newtheorem{definition}[theorem]{Definition}

\newcommand{\Ric}{\mathop{\mathrm{Ric}}}

\def\ppt{\frac{\partial}{\partial t}}

\def\RR{{\mathrm R}}

\def\Rc{{\mathrm {Rc}}}

\def\SS{{\mathrm S}}

\def\Hs{{\mathrm {Hess}}}

\title{Heat kernel estimates under the Ricci-Harmonic map flow}

\author{Mihai B\u{a}ile\c{s}teanu}
\thanks{Department of Mathematics, University of Rochester, 801 Hylan Bld, Rochester, NY 14627, USA \texttt{mbailest@z.rochester.edu}}
\author{Hung Tran}
\thanks{Department of Mathematics, Cornell University, 104 Malott Hall, Ithaca, NY 14853, USA \texttt{htt4@cornell.edu}}

\begin{document}

\begin{abstract}
The paper considers the Ricci flow, coupled with the harmonic map flow between two manifolds. We derive estimates for the fundamental solution of the corresponding conjugate heat equation and we prove an analog of Perelman's differential Harnack inequality. As an application, we find a connection between the entropy functional and the best constant in the Sobolev imbedding theorem in $\mathbb{R}^n$. 
\end{abstract}

\maketitle

\section{Introduction}

Considering two Riemannian manifolds $(M,g)$ and $(N,\gamma)$ and a map $\phi:M\to N$, one defines the Ricci-harmonic map flow as the coupled system of the Ricci flow with the harmonic map flow of $\phi$ given by the following system of equations: \begin{equation}\label{RH_flow_intro}
\begin{cases}\frac\partial{\partial t} g(x,t)=-2\Ric(x,t)+2\alpha(t)\nabla\phi(x,t)\otimes\nabla\phi(x,t)\\
\frac{\partial}{\partial t}\phi(x,t)=\tau_g\phi(x,t)
\end{cases}
 \end{equation} 
$\alpha$ is a positive non-increasing coupling time-dependent function, while $\tau_g\phi$ represents the tension field of the map $\phi$ with respect to the metric $g(t)$. We will call this system the $(RH)_\alpha$ flow and denote as $(g(x,t), \phi(x,t))$ with $t\in[0,T]$ a solution to this flow.  As a result of the coupling, it may be less singular than both the Ricci flow (to which it reduces when $\alpha(t)=0$) and the harmonic map flow. Assuming that the curvature of $M$ remains bounded for all $t\in [0,T]$, we further consider a function $h:M\times [0,T)\times M\times[0,T)\to (0,\infty)$ which is defined implicitly from the following expression $H(x,t;y,T)=(4\pi(T-t)^{-n/2})e^{-h}$, where $n$ is the dimension of $M$ and $H$ is the fundamental solution of the conjugate heat equation: 
\begin{equation}\label{conj_heat_intro}
\Box^*H=\left(-\frac{\partial}{\partial t}-\triangle+S\right)H=0
\end{equation} 
for $S=R-\alpha|\nabla\phi|^2$. The goal of this paper is to study the behavior of the heat kernel $H$ and to prove a Harnack inequality involving the function $h$.

The study of the $(RH)_\alpha$ proves to be useful, since it encompasses in greater generality other flows, for example the Ricci flow on warped produce spaces.  

Historically, it first appeared in \cite{Muller09}, where the author proved short time existence and studied energy and entropy functionals, existence of singularities, a local non-collapsing property etc. His inspiration was a version of this flow, which appeared earlier in the work of List \cite{List08}, where the case of $\phi$ being a scalar function and $\alpha=2$ was analyzed, and where it was shown to be equivalent to the gradient flow of an entropy functional, whose stationary points are solutions to the static Einstein vacuum equations. 

As mentioned above, another instance where the $(RH)_\alpha$ flow arises is when one studies Ricci flow on warped product spaces. More precisely, given a warped product metric $g_M=g_N+e^{2\phi}g_F$ on a manifold $M=N\times F$ (where $\phi\in C^{\infty}(N)$), the Ricci flow equation on $M$: $\frac{\partial g_M}{\partial t}=-2\Ric_M$, leads to the following equations on each component: \[\begin{cases}
\frac{\partial g_N}{\partial t}=-2\Ric_N+2m\ d\phi\otimes\ d\phi \\
\frac{\partial \phi}{\partial t}=\triangle \phi-\mu e^{-2\phi}                                                                                                                                                                                                                                                                                                                                                                                                                                                                                                                                                                                                                                                       \end{cases}\] 
if the fibers $F$ are $m$-dimensional and $\mu$-Einstein. Clearly, this is a particular version of the $(RH)_\alpha$, where the target manifold is one dimensional, and has been studied by M. B. Williams in \cite{Williams13} and by the second author in \cite{Tran12} (when $\mu=0$).  

As in the Ricci flow case, the scalar curvature of a manifold evolving under the $(RH)_\alpha$ flow satisfies the heat equation with a potential (depending on the Ricci curvature of $M$, the map $\phi$ and the Riemann curvature tensor of $N$). Therefore the study of the heat equation and its fundamental solution becomes relevant for understanding of the behavior of the metric under the $(RH)_\alpha$.

A Harnack inequality constitutes a main tool used to study the heat equation, since it compares values at two different points at different times of the solution. A milestone in the field was P. Li and S.-T. Yau's seminal paper  \cite{ly86}, where the authors proved space-time gradient estimates, now called Li-Yau estimates, which, by integration over space-time curves give rise to Harnack inequalities for the heat equation. A matrix version of their result was later proved by Hamilton in \cite{hmatrix93}, who also initiated the study of the heat equation under the Ricci flow \cite{hhrf93,hmatrix93,Hsurvey}. Later this was pursued in \cite{z06,ni04,guenther02,ch09}. Notably, in his proof of the Poincar\'e conjecture, following Hamilton's program, Perelman established in \cite{perelman1} a Li-Yau-Hamilton inequality for the fundamental solution of the conjugate heat equation. Most recently, gradient estimates for the heat equation under the Ricci flow were analyzed in \cite{bcp10}, \cite{liu09} and \cite{JS11}
.   

The technique used in this paper is inspired by Perelman monotonicity formula approach to prove the pseudolocality theorem, and, in particular we derive estimates for the fundamental solution of the corresponding conjugate heat equation, which will lead to a Harnack inequality. 

Since it is relevant to our approach, let's recall Perelman's result. For $(M,g(t)), 0\leq t\leq T$ is a solution to the Ricci flow on an $n$-dimensional closed manifold $M$, define $H=(4\pi(T-t))^{-n/2}e^{-h}$ the fundamental solution of the conjugate heat equation: \[\Box^* H=(-\partial_t-\triangle+R)H=0\]centered at $(y,T)$. Then the quantity \[v=\left((T-t)(2\triangle h-|\nabla h|^2+R)+h-n\right)H\] satisfies $v\leq 0$ for all $t<T$. This inequality proves to be crucial in the study of the functionals developed by Perelman. Moreover, recently X. Cao and Q. Zhang used this inequality in \cite{cz10} to study the behavior of the type I singularity model for the Ricci flow - they proved that, in the limit, one obtains a gradient shrinking Ricci soliton.

The Harnack inequalities that we obtain in our paper are stated in the following theorem and corollary:

\begin{theorem}\label{thm-Harnack}
Let $(\phi(x,t),g(x,t))$, $0\leq t\leq T$ be a solution to (\ref{RH_flow_intro}). Fix $(y,T)$ and let $H=(4\pi(T-t))^{-n/2}e^{-h}$ be the fundamental solution of $(-\partial_t-\triangle+S)H=0$ where $S=R-\alpha|\nabla\phi|^2$. Define $v=\left((T-t)(2\triangle h-|\nabla h|^2+S)+h-n\right)H.$ Then for all $t<T$ the inequality $v\leq 0$ holds true. 
\end{theorem}

\begin{corollary}\label{cor-LYH} Under the above assumptions, let $\gamma(t)$ be a curve on $M$ and $\tau=T-t$. Then the following LYH-type Harnack estimate holds:
\begin{align}\label{LYH}
-\partial_th(\gamma(t),t)&\leq\frac{1}{2}(S(\gamma(t),t)+|\dot\gamma(t)|^2)-\frac{1}{2(T-t)}h(\gamma(t),t)\notag\\
\partial_\tau(2\sqrt{\tau}h)&\leq\sqrt{\tau}((S(\gamma(t),t)+|\dot\gamma(t)|^2)).
\end{align}
\end{corollary}

One may try a different approach to estimate the heat kernel, by means of a Sobolev inequality. This method was used by the first author in \cite{mb10} to bound the heat kernel under the Ricci flow, using techniques developed by Q. Zhang in \cite{z06}, where the Perelman conjugate heat equation was studied. This approach requires fewer conditions on the curvature, and in a particular case, when at the starting time of the flow $S=R-\alpha|\nabla\phi|^2$ is positive, one obtains a bound similar to the one in the fixed metric case. This technique is quite useful, as it connects an analytic invariant (the best constant in the Sobolev imbedding theorem in $\mathbb{R}^n$) to the geometry of the manifold $M$.

The estimates are stated as follows:  
\begin{theorem}\label{theorem}
Let $M^n$ and $N^m$ be two closed Riemannian manifolds, with $n\geq 3$ and let $(g(t),\phi(t)), t\in[0,T]$ be a solution to the $(RH)_\alpha$ flow~\eqref{RH_flow_intro}, with $\alpha(t)$ a non-increasing positive function. Let $H(x,s;y,t)$ be the heat kernel, i.e.  fundamental solution for the heat equation $u_t=\bigtriangleup u$. Then there exists a positive number $C_n$, which depends only on the dimension $n$ of the manifold such that:
\begin{align*}
H(x,s;y,t)\leq &
\frac{C_n}{\left( \int\limits_{s}^{\frac{s+t}{2}}\left(\frac{m_0-c_n\tau}{m_0}\right)^{-2} \frac{e^{\frac{2}{n}F(\tau)}}{A(\tau)} \ d\tau\right)^{\frac{n}{4}} \left(\int\limits_{\frac{s+t}{2}}^{t}
\frac{e^{-\frac{2}{n}F(\tau)}}{A(\tau)} \ d\tau\right)^{\frac{n}{4}}} 
\end{align*}
for $0\leq s<t\leq T$; here $F(t)=\int\limits_{s}^{t}\left[\frac{B(\tau)}{A(\tau)}-\frac{3}{4}\cdot\frac{1}{m_0-c_n\tau}\right] d\tau $,  $1/m_0=\inf_{t=0}S$ - the infimum of $S=R-\alpha|\nabla\phi|^2$, taken at time $0$, and $A(t)$ and $B(t)$ are two positive time functions, which depend on the best constant in the Sobolev imbedding theorem stated above.
\end{theorem}

Notice that there are no curvature assumptions, but $B(t)$ will depend on the lower bound of the Ricci curvature and the derivatives of the curvature tensor at the initial time, as it will follow implicitly from Theorem \ref{thm_Aubin} presented in the section below.  

The estimate may not seem natural, but in a special case, when the scalar curvature satisfies $R(x,0)>\alpha(0)|\nabla\phi(x,0)|^2$, one obtains a bound similar to the the fixed metric case. Recall that J. Wang obtained in \cite{JW97} that the heat kernel on an $n$-dimensional compact Riemannian manifold $M$, with fixed metric, is bounded from above by $N(S)(t-s)^{-n/2}$, where $N(S)$ is the Neumann Sobolev constant of $M$, coming from a Sobolev imbedding theorem. Our corollary exhibits a similar bound:   

\begin{corollary}\label{Sob-cor}
Under the same assumptions as in theorem (\ref{theorem}), together with the condition that $R(x,0)>\alpha(0)|\nabla\phi(x,0)|^2$, there exists a positive number $\tilde{C}_n$, which depends only on the dimension $n$ of the manifold and on the best constant in the Sobolev imbedding theorem in $\mathbb{R}^n$, such that:
\begin{align*}
H(x,s;y,t)\leq \tilde{C}_n\cdot\frac{1}{(t-s)^{\frac{n}{2}}} \hspace{1cm} \text{ for } 0\leq s<t\leq T
\end{align*}
\end{corollary}

The exact expression of $\tilde{C}_n$ is $\left(\frac{4K(n,2)}{n}\right)^{\frac{n}{2}}$, where $K(n,2)$ is the best constant in the Sobolev imbedding in $\mathbb{R}^n$. Let's note that this result, in fact, improves the result in \cite{mb10}, since in this case the constants are sharper.

As an application, we can prove the following theorem, connecting the functional $\mathbb{W}_\alpha$ (which is analogous to Perelman's entropy functional) to the best constant in the Sobolev imbedding: 

\begin{theorem}\label{comp-theorem}
Let $(\phi(x,t),g(x,t))$, $0\leq t\leq T$ be a solution to $(RH)_\alpha$ flow and let $\mathbb{W}_\alpha$ be the entropy functional defined in section \ref{prelim}. If $\mu_\alpha$ is the associated functional $\mu_\alpha(g,\phi,\tau)=\inf\limits_f\mathbb{W}_\alpha(g,\phi,\tau, f)$, then 
\[\mu_\alpha(g,\phi,\tau)\geq \frac{\tau D}{3}\ln\left[(4\pi)^{n/2}\tilde{C}_n\right]\]
where $D=\inf_{M\times \{0\}}{S}$ and $\tilde{C}_n=\left(\frac{4K(n,2)}{n}\right)^{\frac{n}{2}}$.
\end{theorem}

The paper is organized as follows: in section \ref{prelim} we introduce the notation and explain the setting for our problem, while in section \ref{grad-est} we prove some lemmas and proposition needed in the proof of the Harnack inequalities. We continue with section \ref{main-results} which presents the proof of theorem \ref{thm-Harnack} and its corollary. We then present in section \ref{doi} the Sobolev imbedding theorems used in the proof of theorem \ref{theorem}, while its proof, together with the corollary and proof of theorem \ref{comp-theorem}  are presented in section \ref{patru}. 

\subsection{Acknowledgments} Both authors want to express their gratitude to Prof. Xiaodong Cao for fruitful discussions and suggestions about this project. 

\section{Preliminaries}\label{prelim}

We present a review of the basic equations and identities for the $(RH)_\alpha$ flow, together with the more detailed setting of the problem. 

Consider $(M^n,g)$ and $(N^m, \gamma)$ two $n$-dimensional and $m$-dimensional, respectively, manifolds without boundary, which are compact, connected, oriented and smooth. We also let $g(t)$ be a family of Riemannian metrics on $M$, while $\phi(t)$ a family of smooth maps between $M$ and $N$. We assume that $N$ is isometrically embedded into the Euclidean space $\mathbb{R}^d$ (which follows by Nash's embedding theorem) for large enough $d$, so one may write $\phi=(\phi^\mu)_{1\leq\mu\leq d}$.    

For $T>0$, denote with $(g(t),\phi(t)), t\in[0,T]$ a solution to the following coupled system of Ricci flow and harmonic map flow, i.e. the $(RH)_\alpha$ flow,  with coupling time-dependent constant $\alpha(t)$: \begin{equation}\label{RH_flow}
\begin{cases}\frac\partial{\partial t} g(x,t)=-2\Ric(x,t)+2\alpha(t)\nabla\phi(x,t)\otimes\nabla\phi(x,t)\\
\frac{\partial}{\partial t}\phi(x,t)=\tau_g\phi(x,t)
\end{cases}
 \end{equation} 
 
The tensor $\nabla\phi(x,t)\otimes\nabla\phi(x,t)$ has the following expression in local coordinates: $(\nabla\phi\otimes\nabla\phi)_{ij}=\nabla_i\phi^\mu\nabla_j\phi^\mu$ and the energy density of the map $\phi$ is given by $|\nabla\phi|^2=g^{ij}\nabla_i\phi^\mu\nabla_j\phi^\mu$, where we use the convention (from \cite{Muller09}) that repeated Latin indices are summed over from $1$ to $n$, while the Greek are summed from 1 to $d$. All the norms are taken with respect to the metric $g$ at time $t$. 
 
We assume the most general condition for the coupling function $\alpha(t)$, as it appears in \cite{Muller09}: it is a non-increasing function in time, bounded from below by $\bar\alpha>0$, at any time.  
  
We choose a small enough $T>0$, such that a solution to this system exists in $[0,T]$ (R. M\"uller proved the short time existance of the flow in \cite{Muller09}, so we just pick $T<T_\epsilon$, where $T_\epsilon$ is the moment where there is possibly a blowup). 

Following the notation in \cite{Muller09}, it will be easier to introduce these quantities: 
\begin{align*}
\mathcal{S}&:= \Ric-\alpha\nabla\phi\otimes\nabla\phi\\
S_{ij}&:=R_{ij}-\alpha\nabla_i\phi\nabla_j\phi\\
S&:=R-\alpha|\nabla\phi|^2
\end{align*}

\subsection{Heat kernel under $(RH)_\alpha$ flow}

Our proof will focus on obtaining bounds on the heat kernel $H(x,s;y,t)$, which is the fundamental solution of the heat equation
\begin{align}\label{heateqn}
\left(\triangle-\frac\partial{\partial t}\right)u(x,t)=0,\qquad x\in M,~t\in[0,T].
\end{align}
Such heat kernel does indeed exist and it is well defined, as it was shown in \cite{guenther02} by C. Guenther, who studied the fundamental solution of the linear parabolic operator $L(u)=(\triangle-\frac{\partial}{\partial t}-f)u$, on compact n-dimensional manifolds with time dependent metric, where $f$ is a smooth space-time function. She proved the uniqueness, positivity, the adjoint property and the semigroup property of this operator, which thus behaves like the usual heat kernel. As a particular case ($f=0$), she obtained the existence and properties of the heat kernel under any flow of the metric.

Given a linear parabolic operator $L$, its fundamental solution $H(x,s;y,t)$ is a smooth function $H(x,s;y,t):M\times[0,T]\times M\times[0,T]\to\mathbb{R}$, with $s<t$, which satisfies two properties: 
\begin{enumerate}
\item[(i)] $L(H)=0$ in $(y,t)$ for $(y,t)\neq (x,s)$ 
\item[(ii)] $\lim_{t\to s}H(x,s;.,t)=\delta_x$ for every $x$, where $\delta_x$ is the Dirac delta function.
\end{enumerate} 

In our case, $L$ is the heat operator, so $H$ satisfies the heat equation in the $(y,t)$ coordinates \[\triangle_y H(x,s;y,t)-\partial_t H(x,s;y,t)=0\] whereas in the $(x,s)$ it satisfies the adjoint or conjugate heat equation \[\triangle_x H(x,s;y,t)+\partial_s H(x,s;y,t)- [R(x,s)-\alpha|\nabla\phi|^2] H(x,s;y,t)=0\] or \[\triangle_x H(x,s;y,t)+\partial_s H(x,s;y,t)- S(x,s) H(x,s;y,t)=0\] (see \cite{Muller09} for a proof of this fact), where $R(x,s)$ is the scalar curvature, measured with respect to the metric $g(s)$.

We therefore denote with $\Box^*=-\partial_t-\triangle+S$ the adjoint heat operator, adapted to the $(RH_\alpha)$ flow. 

We fix $T$, which is the final time of the flow and we look backwards in time, so we introduce the backward time $\tau=T-t>0$ and we consider another function $h:M\times [0,T)\times M\times[0,T)\to (0,\infty)$ which is defined implicitly as \[H(x,t;y,T)=(4\pi(T-t)^{-n/2})e^{-h}=(4\pi\tau)^{-n/2}e^{-h},\] where $n$ is the dimension of $M$. This function $h$ is the center of our investigation.    

As $t\rightarrow T$, the heat kernel exhibits an asymptotic behavior, as one can see from the following theorem, which was proven for the Ricci flow, but the arguments can be applied verbatim to the $(RH)_\alpha$ flow.   
\begin{theorem}\cite[Theorem 24.21]{chowetc3} \label{asymptoticconj}
For $\tau=T-t$,
\begin{equation*} 
H(x,t;y,T) \sim \frac{e^{-\frac{d_{T}^{2}(x,y)}{4\tau}}}{(4\pi\tau)^{n/2}}\Sigma_{j=0}^{\infty}\tau^{j}u_{j}(x,y,\tau).
\end{equation*}
More precisely, there exists $t_{0}>0$ and a sequence $u_{j}\in C^{\infty}(M\times M\times [0,t_{0}])$ such that,
\begin{equation*}
H(x,t;y,T)-\frac{e^{-\frac{d_{T}^{2}(x,y)}{4\tau}}}{(4\pi\tau)^{n/2}}\Sigma_{j=0}^{k}\tau^{j}u_{j}(x,y,T-l)=w_{k}(x,y,\tau),
\end{equation*} 
with 
\begin{equation*}
u_{0}(x,x,0)=1,
\end{equation*}
and 
\begin{equation*}
w_{k}(x,y,\tau)=O(\tau^{k+1-\frac{n}{2}})
\end{equation*}
as $\tau\rightarrow 0$ uniformly for all $x,y\in M$.
\end{theorem}

\subsection{The Entropy functional}

Next, we recall the $\mathbb{W}_{\alpha}$ entropy functional, as it was defined in \cite{Muller09}, since it will be used in our future proofs. 
\begin{definition}
Along the $(RH)_\alpha$ flow given by (\ref{RH_flow}), one defines the entropy functional restricted to functions $f$ satisfying $\int_{M}(4\pi\tau)^{-n/2}e^{-f}d\mu_{M}=1$ as
\begin{equation}
\mathbb{W}_{\alpha}(g,\phi,\tau,f)=\int_{M}\Big(\tau(|\nabla f|^2+\SS)+(f-n)\Big)(4\pi\tau)^{-n/2}e^{-f}d\mu_{M}.
\end{equation}
There are two more associated functionals, which are defined similarly as follows:
\begin{align}
\mu_{\alpha}(g,\phi,\tau)&=\inf_{f}{\mathbb{W}_{\alpha}(g,\phi,\tau,f)},\\
\upsilon_{\alpha}(g,\phi)&=\inf_{\tau>0}{\mu_{\alpha}(g,\phi,\tau)}.
\end{align}
\end{definition}
\begin{remark} It is trivial to show that these functionals are invariant under diffeomorphisms and scaling:
\begin{align*}
\mathbb{W}_{\alpha}(g,\phi,\tau,f)&=\mathbb{W}_{\alpha}(cg,\phi,c\tau,f),\\
\mu_{\alpha}(g,\phi,\tau)&=\mu_{\alpha}(cg,\phi,c\tau),\\
\upsilon_{\alpha}(g,\phi)&=\upsilon_{\alpha}(cg,\phi).
\end{align*}
\end{remark}

We next present some lemmas whose proofs are identical to the counterpart for the Ricci flow.

\begin{lemma} 
\label{basicmu}
We consider a closed Riemannian manifold $(M,g)$, smooth function $\phi: M\mapsto N$ and $\tau>0$.
 
{\bf a.} Along the flow (\ref{RH_flow}), with $\alpha(t)\equiv \alpha>0$, $\tau(t)>0$, $\frac{d}{dt}\tau=-1$ then $\mathbb{W}_{\alpha}(g,\phi,\tau,f)$ is non-decreasing in time t.

{\bf b.} There exists a smooth minimizer $f_{\tau}$ for $\mathbb{W}_\alpha(g,\phi,\tau,.)$ which satisfies
\[\tau(2\triangle{f_{\tau}}-|\nabla{f_{\tau}}|^2+\SS)+f_{\tau}-n=\mu_\alpha(g,u,\tau).\]
In fact, $\mu_{\alpha}(g,u,\tau)$ is finite.

{\bf c.} Along the flow (\ref{RH_flow}), with $\alpha(t)\equiv \alpha>0$, $\tau(t)>0$, $\frac{d}{dt}\tau=-1$ then $\mu_{\alpha}(g,\phi,\tau)$ is non-decreasing in time t.

{\bf d.} $\lim_{\tau\rightarrow 0^{+}}\mu_{\alpha}(g,u,\tau)=0$.
\end{lemma}
\begin{proof}
Part (a), (b,c) follow from \cite[Props 7.1, 7.2]{Muller09} respectively.

The proof of part (d) is almost identical to that of \cite[Prop 3.2]{stw03notes} (also \cite[Prop 17.19, 17.20]{chowetc3}) so we give a brief argument here.

First, by the scaling invariance, 
\[\mathbb{W}_{\alpha}(g,\phi, \tau, f)=\mathbb{W}_{\alpha}\left(\frac{1}{\tau}g,\phi, 1, f\right).\]

In addition, we have $\SS_{\frac{1}{\tau}g}=\tau\SS_{g}$ and we can construct a test function such that
\[ \lim_{\tau\rightarrow 0} \mathbb{W}_\alpha(g(\tau),\phi(\tau), \tau, f(\tau))=0. \]

Finally, the equality follows from a contradiction with the Gross\rq{}s logarithmic Sobolev inequality on an Euclidean space using a blow-up argument as
$\frac{1}{\tau}g$ converges to the Euclidean metric.
\end{proof}

Let's now note an identity that is essential for our future computations:
\begin{lemma}
 Along the flow (\ref{RH_flow_intro}), with $\alpha(t)\geq0$ and non-increasing, we have,
\begin{equation}
\label{evolS}
\frac{\partial}{\partial t}\SS =\triangle \SS+2\alpha|\tau_g\phi|^2+2|\mathcal{S}_{ij}|^2-\alpha'(t)|\nabla\phi|^2.
\end{equation}
\end{lemma}
\begin{proof}
See \cite[Theorem 4.4]{Muller09}.
\end{proof}

\subsection{The $\mathcal{L}_\phi$-length of a curve}
\begin{definition}
Given $\tau(t)=T-t$, define the $\mathcal{L}_{\phi}$-length of a curve $\gamma: [\tau_{0},\tau_{1}]\mapsto M$, $[\tau_{0},\tau_{1}]\subset[0,T]$ by 
\begin{equation}
\mathcal{L}_{\phi}(\gamma):=\int_{\tau_{0}}^{\tau_{1}}\sqrt{\tau}\left(S(\gamma(\tau))+|\dot{\gamma}(\tau)|^2\right)d\tau.
\end{equation}
For a fixed point $y\in M$ and $\tau_{0}=0$, the backward reduced distance is defined as 
\begin{equation}
\label{adaptedRD}
\ell_{\phi}(x,\tau_{1}):=\inf_{\gamma\in \Gamma}\left\{\frac{1}{2\tau_{1}}\mathcal{L}_{\phi}(\gamma)\right\},
\end{equation}
where $\Gamma=\{\gamma:[0,\tau_{1}]\mapsto M, \gamma(0)=y, \gamma(\tau_{1})=x\}$.\\
Finally, the backward reduced volume is defined as:  
\begin{equation}
V_{\phi}(\tau):=\int_{M}(4\pi\tau)^{-n/2}e^{-\ell_{\phi}(y,\tau)}d\mu_{\tau}(y).
\end{equation}
\end{definition}

We conclude this section with a technical lemma, that will prove a useful bound for the heat kernel in terms of the reduced distance. We will only sketch the proof, as the arguments are standard. 

\begin{lemma}
\label{compareRDandconjheat}
Define $L_{\phi}(x,\tau)=4\tau\ell_{\phi}(x,\tau)$. Then the following hold: \\
{\bf a.} Assume that there are $k_{1},k_{2}\geq 0$ such that $-k_{1}g(t)\leq \mathcal{S}(t) \leq k_{2}g(t)$ for $t\in[0,T]$. Then $L_{\phi}$ is smooth almost everywhere and a locally Lipschitz function on $M\times [0,T]$. Moreover, 
$$e^{-2k_{1}\tau}d_{T}^2(x,y)-\frac{4k_{1}n}{3}\tau^2\leq L_{\phi}(x,\tau)\leq e^{2k_{2}\tau}d_{T}^2(x,y)+\frac{4k_{2}n}{3}\tau^2 .$$
{\bf b.} $\Box^{\ast}\Big(\frac{e^{-\frac{L_{\phi}(x,\tau)}{4\tau}}}{(4\pi\tau)^{n/2}}\Big)\leq 0.$\\
{\bf c.} $H(x,t;y,T)=(4\pi\tau)^{-n/2}e^{-h}$ then $h(x,t;y,T)\leq \ell_{\phi}(x,T-t)$ . 
\end{lemma}
\begin{proof}
{\bf a.} This is a direct consequence of  \cite[Lemma 4.1]{Muller10} for general flows.\\
{\bf b.} The result follows from \cite[Lemma 5.15]{Muller10}, where the key of the proof is given by the non-negativity of the quantity,
\begin{equation*}
\mathcal{D}(\mathcal{S},X)=\partial_{t}S-\triangle{S}-2|\mathcal{S}|^2+4(\nabla_{i}\mathcal{S}_{ij})X_{j}-2(\nabla_{j}S)X_{j}+2(\text{Rc}-\mathcal{S})(X,X).
\end{equation*} 
In our case, applying (\ref{evolS}) and the identity $4(\nabla_iS_{ij})X_j-2(\nabla_jS)X_j=-4\alpha\tau_g\phi\nabla_j\phi X_j$ (a generalized second Bianchi identity) yields
\begin{align*}
\mathcal{D}(\mathcal{S},X)&=2\alpha|\tau_g\phi(x,t)|^2+4\nabla^{i}S_{ij}X^{j}-2\nabla_{j}SX^{j}+2\alpha\nabla_i\phi\nabla_j\phi X^iX^j-\alpha'(t)|\nabla\phi|^2\\
&=2\alpha|\tau_g\phi-\nabla_X\phi|^2-\alpha'(t)|\nabla\phi|^2 \geq 0.
\end{align*}
assuming that $\alpha(t)>0$ is non-increasing. 

{\bf c.} A detailed argument for this inequality can be found in \cite[Lemma 16.49]{chowetc2}. First observe that part a) implies $\lim_{\tau\rightarrow 0}L_{\phi}(x,\tau) =d_{T}^2(y,x)$ and, hence,
\[\lim_{\tau\rightarrow 0}\frac{e^{-\frac{L_{\phi}(x,\tau)}{4\tau}}}{(4\pi\tau)^{n/2}}=\delta_{y}(x),\]
since locally a Riemannian manifold looks like the Euclidean space. Using part b) and the maximum principle one obtains that 
\begin{equation*}
H(x,t;y,T)\geq \frac{e^{-\frac{L_{\phi}(x,\tau)}{4\tau}}}{(4\pi\tau)^{n/2}}=\frac{e^{-\frac{L_{\phi}(x,T-t)}{4\tau}}}{(4\pi(T-t))^{n/2}}.
\end{equation*}
Finally, one concludes that
\begin{equation*}
h(x,t;y,T)\leq \frac{L_{\phi}(x,\tau)}{4\tau}=\ell_{\phi}(x,\tau)=\ell_{\phi}(x,T-l).
\end{equation*}
\end{proof}

\section{Heat kernel and gradient estimates}\label{grad-est}

Having presented the background of our problem and introduced the notation, we are now ready to prove some results that will lead to the proof of theorem \ref{thm-Harnack}.  

First we deduce a general estimate on the heat kernel, inspired by the proof in the Ricci flow case in \cite{cz11}. 

\begin{lemma}
\label{conjestimateS}
Let $B =-\inf\limits_{0<\tau\leq T} \mu_{\alpha}(g,\phi,\tau)$  ($B$ is well-defined as proven in \cite{Muller09}) and $D =\min\{0,\inf_{M\times \{0\}}{S}\}$. Then the following inequality holds $$H(x,t;y,T) \leq e^{B-(T-t)D/3}(4\pi (T-t))^{-n/2}.$$
\end{lemma}

\begin{proof}
We may assume without loss of generality that $t=0$. Denote with $\Phi(y,t)$ a positive solution to the heat equation along the $(RH)_\alpha$ flow. First, we obtain an upper bound for the $L^{\infty}$-norm of $\Phi(.,T)$ in terms of $L^{1}$-norm of $\Phi(.,0)$.\\
Set $p(l)=\frac{T}{T-l}=\frac{T}{\tau}$ then $p(0)=1$ and $\lim_{l\rightarrow T}p(l)=\infty$. For $A=\sqrt{\int_{M}\Phi^{p}d\mu}$, $v=A^{-1}\Phi^{p/2}$ and $ \nabla \Phi\nabla (v^2\Phi^{-1})=(p-1)p^{-2}4|\nabla v|^2$, integration by parts (IBP) yields 
\begin{align*}
\partial_{t}(\ln{||\Phi||_{L^{p}}})&=-p'p^{-2}\ln(\int_{M}\Phi^{p}d\mu)+(p\int_{M}\Phi^pd\mu)^{-1}\partial_{t}(\int_{M}\Phi^{p}d\mu)
\\&
=-p'p^{-2}\ln(\int_{M}\Phi^{p}d\mu)+(p\int_{M}\Phi^pd\mu)^{-1}\Big(\int_{M}\Phi^{p}(p\Phi^{-1}\Phi'+p'\ln{\Phi}-S)d\mu\Big)
\\&
=-p'p^{-2}\ln(A^2)+p^{-1}A^{-2}\Big(\int_{M}A^2v^2(p\Phi^{-1}\Phi'+p'\frac{2}{p}\ln{(Av)}-S)d\mu\Big)  
\\&
=\int_{M}v^2\Phi^{-1}\triangle \Phi d\mu+p'p^{-2}\int v^2\ln{v^2}-p^{-1}\int_{M} Sv^2d\mu
\\&
=p'p^{-2}\int_{M} v^2\ln{v^2}d\mu-(p-1)p^{-2}\int_{M}4|\nabla v|^2 d\mu-p^{-1}\int_{M} Sv^2d\mu
\\&
=p'p^{-2}\Big(\int_{M}v^2\ln{v^2}d\mu-\frac{p-1}{p'}\int_{M}4|\nabla v|^2d\mu-\frac{p-1}{p'}\int_{M}Sv^2d\nu\Big)
\\&
+((p-1)p^{-2}-p^{-1})\int_{M}Sv^2d\mu.
\end{align*} 
Setting $v^2=(4\pi\tau)^{-n/2}e^{-h}$ then the first term becomes 
$$-p'p^{-2}\mathbb{W}_{\alpha}(g,u,\frac{p-1}{p'},h)-n-\frac{n}{2}\ln({4\pi\frac{p-1}{p'}}).$$
Notice that
\[ p'p^{-2}=\frac{1}{T},  \frac{p-1}{p'}=\frac{l(T-l)}{T},\mbox{ and } (p-1)p^{-2}-p^{-1}=-\frac{(T-l)^2}{T^2}.\]
For $0<t_{0}<T$, $\tau(t_{0})=\frac{t_{0}(T-t_{0})}{T}$ and $\frac{d}{dt}{\tau}=-1$ then $0<\tau(0)=\frac{t_{0}(2T-t_{0})}{T}<T$.
Using Lemma \ref{basicmu}, we find that
$$-p'p^{-2}\mathbb{W}_{\alpha}(g(l),u,\frac{p-1}{p'},h)\leq -\frac{1}{T}\mathbb{W}_{\alpha}(g(0),u,\tau(0),h)\leq -\frac{1}{T}\inf_{0<\tau\leq T} \mu_{\alpha}(g(0),\tau)=\frac{B}{T}.$$
Therefore $$T\partial_{t}(\ln{|\Phi||_{L^{p}}})\leq B-n-\frac{n}{2}\ln{(4\pi\frac{t(T-t)}{T})}-\frac{(T-t)^2}{T} D,$$ since, by (\ref{evolS}), the minimum of S is nondecreasing along the flow.
By integrating the above inequality one obtains 
$$T\ln{\frac{||\Phi(.,T)||_{L^{\infty}}}{||\Phi(.,0)||_{L^{1}}}}\leq T(B-n-\frac{n}{2}(\ln{(4\pi T)}-2))-\frac{T^{2}}{3}D.$$
Then
$$||\Phi(.,T)||_{L^{\infty}}\leq e^{B-TD/3}(4\pi T)^{-n/2}||\Phi(.,0)||_{L^{1}}.$$
By the definition of the heat kernel: 
\begin{equation}
\label{heatbyheatkernel}
\Phi(y,T)=\int_{M}H(x,0,y,T)\Phi(x,0)d\mu_{g(0)}(x),
\end{equation}
so, together with the fact that the above inequality holds for any arbitrary positive solution to the heat equation, we obtain 
$$H(x,0,y,T) \leq e^{B-TD/3}(4\pi T)^{-n/2}.$$
\end{proof}

The next result is a gradient estimate for the solution of the adapted conjugate heat equation.

\begin{lemma}
\label{gradconjS}Assume there exist $k_{1},k_{2},k_{3}, k_4>0$ such that the followings hold on $M\times[0,T]$, 
\begin{align*}
-\Rc (g(t)) &\leq k_{1}g(t),\\
-\mathcal{S} &\leq k_{2} g(t),\\
|\nabla{\SS}|^2 &\leq k_{3},\\
|\SS| & \leq k_{4}. 
\end{align*}
Let q be any positive solution to the equation $\Box^{\ast}q=0$ on $M\times [0,T]$, and $\tau=T-t$. If $q<A$ then there exist $C_{1},C_{2}$ depending on $k_{1},k_{2},k_{3}, k_{4}$ and n such that for $0<\tau\leq \min\{1,T\}$, we have
\begin{equation}
\tau \frac{|\nabla{q}|^2}{q^2}\leq (1+C_{1}\tau)(\ln{\frac{A}{q}}+C_{2}\tau).
\end{equation}
\end{lemma}

\begin{proof}
We start by computing
\begin{align}\label{conj_heat_op}
\left(-\ppt-\triangle\right)\frac{|\nabla{q}|^2}{q}&=\SS\frac{|\nabla{q}|^2}{q}+\frac{1}{q}\left(-\ppt-\triangle\right)|\nabla{q}|^2-2\frac{|\nabla{q}|^4}{q^3}+2\nabla(|\nabla{q}|^2)\frac{\nabla q}{q^2}.
\end{align}
Observing that \[\nabla(|\nabla{q}|^2)\nabla q=\nabla(g^{ij}\nabla_iq\nabla_jq)\nabla q =2 g^{ij}\nabla(\nabla_iq)\nabla_jq\cdot\nabla q=2\nabla^2q(\nabla q,\nabla q)\]
while
\begin{align*}
\triangle(|\nabla q|^2)&=2|\nabla^2 q|^2+2R_{ij}\nabla_iq\nabla_jq+2\nabla_iq\nabla_i(\triangle q)\\
\ppt(|\nabla q|^2)&= 2S_{ij}\nabla_iq\nabla_jq+2\nabla q\cdot\nabla \left(\ppt q\right) 
\end{align*}
one can turn the second term in (\ref{conj_heat_op}) into 
\begin{align*}
\frac{1}{q}\left(-\ppt-\triangle\right)|\nabla{q}|^2&=\frac{1}{q}\Big[-2(\mathcal{S}+\Rc)(\nabla{q},\nabla{q})-2\nabla{q}\nabla{(\SS q)}-2|\nabla^2{q}|^2 \Big].
\end{align*}
Thus equation (\ref{conj_heat_op}) now becomes:
\begin{align*}
\left(-\ppt-\triangle\right)\frac{|\nabla{q}|^2}{q}&=\frac{-2}{q}\left(|\nabla^2{q}|^2-2\frac{1}{q}\nabla^2q(\nabla q,\nabla q)+\frac{|\nabla q|^4}{q^2}\right)^2 \\
    &+ \frac{-2(\mathcal{S}+\Rc)(\nabla{q},\nabla{q})-2\SS\nabla{q}\nabla{q}-2q\nabla{q}\nabla{\SS}}{q}+\SS\frac{|\nabla{q}|^2}{q}\\
    &=\frac{-2}{q}\left|\nabla^2{q}-\frac{\nabla q\otimes \nabla q}{q}\right|^2+\\
    &+\frac{-2(\mathcal{S}+\Rc)(\nabla{q},\nabla{q})-2\SS\nabla{q}\nabla{q}-2q\nabla{q}\nabla{\SS}}{q}+\SS\frac{|\nabla{q}|^2}{q}\\
   &\leq \frac{-2(\mathcal{S}+\Rc)(\nabla{q},\nabla{q})-2\SS\nabla{q}\nabla{q}-2q\nabla{q}\nabla{\SS}}{q}+\SS\frac{|\nabla{q}|^2}{q}\\ 
   &\leq (2(k_{1}+k_{2})+nk_2)\frac{|\nabla{q}|^2}{q}+2|\nabla{q}||\nabla{\SS}|\\
   &\leq (2k_{1}+(2+n)k_2+1)\frac{|\nabla{q}|^2}{q}+k_{3}q
\end{align*}
where we have used the assumption that there exist $k_{1},k_{2},k_{3}, k_4>0$ such that $-\Rc (g(t)) \leq k_{1}g(t)$, $-\mathcal{S} \leq k_{2} g(t)$, $|\nabla{\SS}|^2 \leq k_{3}$ and $|\SS|  \leq k_{4}$.  

Furthermore, we have
\begin{align*}
\left(-\ppt-\triangle\right)\left(q\ln{\frac{A}{q}}\right)&=-\SS q\ln{\frac{A}{q}}+\SS q+\frac{|\nabla{q}|^2}{q}\\
& \geq \frac{|\nabla{q}|^2}{q}-nk_{2}q-k_{4}q\ln{\frac{A}{q}}.
\end{align*}
Let $\Phi(x,\tau)=a(\tau)\frac{|\nabla{q}|^2}{q}-b(\tau)q\ln{\frac{A}{q}}-c(\tau)q,$ then 
\begin{align*}
\left(-\ppt-\triangle\right)\Phi \leq &\frac{|\nabla{q}|^2}{q}\Big(a\rq{}(\tau)+a(\tau)(2k_{1}+(2+n)k_2+1)-b(\tau)\Big)\\
&+q\ln{\frac{A}{q}}\Big(k_{4}b(\tau)-b\rq{}(\tau)\Big)\\
&+q\Big(k_{3} a(\tau)-c\rq{}(\tau)+nk_{2}b(\tau)+c(\tau)k_{4}\Big).
\end{align*}
We are free to choose the functions $a,b,c$ appropriately such that $(-\partial_{t}-\triangle)\Phi\leq 0$. For example,
\begin{align*}
a(\tau)&=\frac{\tau}{1+(2k_{1}+(2+n)k_2+1)\tau},\\
b(\tau)&=e^{k_{4}\tau},\\
c(\tau)&=(e^{k_{5}k_{4}\tau}nk_{2}+k_{3})\tau.
\end{align*}
Denote with $k_{5}=1+\frac{k_3}{nk_2}$. By the maximum principle, noticing that $\Phi\leq 0$ at $\tau=0$, 
$$a\frac{|\nabla{q}|^2}{q}\leq b(\tau)q\ln{\frac{A}{q}}+cq.$$

Then, one can conclude that there exist $C_{1},C_{2}$ depending on $k_{1},k_{2},k_{3}, k_{4}$ and $n$ such that for $0<\tau\leq \min\{1,T\}$, we have
\begin{equation}
\tau \frac{|\nabla{q}|^2}{q^2}\leq (1+C_{1}\tau)\left(\ln{\frac{A}{q}}+C_{2}\tau\right).
\end{equation}

\end{proof}

Finally, we will need the following lemma, where the $l_\phi$ distance, introduced in the second section, will be used. 

\begin{lemma}
\label{integralboundabove} Using the notation as in the previous lemma, the following inequality holds $\int_{M}hH\Phi d\mu_{M}\leq \frac{n}{2}\Phi(y,T)$, i.e, $\int_{M}(h-\frac{n}{2})H\Phi d\mu_{M}\leq 0.$
\end{lemma}

\begin{proof}
By lemma \ref{compareRDandconjheat} we have
\begin{align*}
\limsup_{\tau\rightarrow 0}\int_{M}hH\Phi d\mu_{M}\leq \limsup_{\tau\rightarrow 0}\int_{M} \ell_{w}(x,\tau)H\Phi d\mu_{M}(x)\\
\leq \limsup_{\tau\rightarrow 0}\int_{M} \frac{d_{T}^2(x,y)}{4\tau}H\Phi d\mu_{M}(x).
\end{align*}
Using Lemma \ref{asymptoticconj}, 
\begin{equation*}
\lim_{\tau\rightarrow 0}\int_{M} \frac{d_{T}^2(x,y)}{4\tau}H\Phi d\mu_{M}(x)=\lim_{\tau\rightarrow 0}\int_{M} \frac{d_{T}^2(x,y)}{4\tau}\frac{e^{-\frac{d_{T}^2(x,y)}{4\tau}}}{(4\pi\tau)^{n/2}}\Phi d\mu_{M}(x).
\end{equation*}
Either by differentiating twice under the integral sign or using these following standard identities on Euclidean spaces
\[\int_{-\infty}^{\infty}e^{-a\textbf{x}^2}d\textbf{x} =\sqrt{\frac{\pi}{a}} \text{ and } \int_{-\infty}^{\infty}\textbf{x}^{2}e^{-a\textbf{x}^2}d\textbf{x} =\frac{1}{2a}\sqrt{\frac{\pi}{a}},
\]
we find
\begin{equation*}
\int_{\RR^{n}}|x|^{2}e^{-a|x|^2}dx=n\left(\int_{-\infty}^{\infty}\textbf{x}^{2}e^{-a\textbf{x}^2}d\textbf{x}\right)\left(\int_{-\infty}^{\infty}e^{-a\textbf{x}^2}d\textbf{x}\right)^{n-1}=\frac{n}{2a}\left(\frac{\pi}{a}\right)^{n/2}.
\end{equation*}
Therefore,
\begin{equation*}
\lim_{\tau\rightarrow 0}\frac{d_{T}^2(x,y)}{4\tau}\frac{e^{-\frac{d_{T}^2(x,y)}{4\tau}}}{(4\pi\tau)^{n/2}}=\frac{n}{2}\delta_{y}(x)
\end{equation*}
and so 
\begin{equation*}
\lim_{\tau\rightarrow 0}\int_{M} \frac{d_{T}^2(x,y)}{4\tau}\frac{e^{-\frac{d_{T}^2(x,y)}{4\tau}}}{(4\pi\tau)^{n/2}}\Phi d\mu_{M}(x)=\frac{n}{2}\Phi(y,T).
\end{equation*}
The result now follows.
\end{proof}

\section{Proof of Theorem \ref{thm-Harnack}}\label{main-results}

The procedure will be standard, we will apply the maximum principle. In order to do that, we need to prove the non-positivity of $\Box^*v$. 

\subsection{Evolution of the Harnack Quantity}
\begin{lemma}
Let $v=\Big((T-t)(2\triangle{h}-|\nabla{h}|^2+S)+h-n\Big)H$ then 
\begin{equation}
\Box^{\ast}v=-2(T-t)\Big(\left|\mathcal{S}+\Hs h-\frac{g}{2\tau}\right|^2+2\alpha(\langle\nabla\phi,\nabla h\rangle^2+|\tau_g\phi|^2)\Big)H\leq 0.
\end{equation}
\end{lemma}
\begin{proof}

Let $q=2\triangle{h}-|\nabla{h}|^2+S$ then 
\begin{align*}
H^{-1}\Box^{\ast}v&=-(\partial_{t}+\triangle)(\tau q+h)-2\left\langle{\nabla(\tau q+h),H^{-1}\nabla{H}}\right\rangle\\
&=q-\tau(\partial_{t}+\triangle)q-(\partial_{t}+\triangle)h+2\tau\left\langle{\nabla q,\nabla{h}}\right\rangle+2|\nabla{h}|^2.
\end{align*}
As H satisfies $\Box^{\ast}H=0$, $(\partial_{t}+\triangle)h=-S+|\nabla{h}|^2+\frac{n}{2\tau}$.
We compute
\begin{align*}
(\partial_{t}+\triangle)\triangle{h}&=\triangle\frac{\partial{h}}{\partial{t}}+2S_{ij}\nabla_i\nabla_jh+\triangle(\triangle h)\\
&=\triangle\left(-\triangle{h}+|\nabla{h}|^2-S+\frac{n}{2\tau}\right)+\triangle(\triangle h)+2\left\langle{\mathcal{S},\Hs(h)}\right\rangle\\
&=\triangle(|\nabla{h}|^2-S)+2\left\langle{\mathcal{S},\Hs(h)}\right\rangle,
\end{align*}
where we used the formula for the evolution of the Laplacian under the $(RH)_\alpha$ flow.
\begin{align*}
(\partial_{t}+\triangle)|\nabla{h}|^2=&2\mathcal{S}(\nabla{h},\nabla{h})+2\left\langle{\nabla{h},\nabla{\frac{\partial{h}}{\partial{t}}}}\right\rangle+\triangle{|\nabla{h}|^2}\\
=&2\left\langle{\nabla{h},\nabla(-\triangle{h}+|\nabla{h}|^2-S)}\right\rangle\\
&+2\mathcal{S}(\nabla{h},\nabla{h})+\triangle{|\nabla{h}|^2}.
\end{align*}
Recall from (\ref{evolS}), $(\partial_{t}+\triangle)S=2\triangle{S}+2|\mathcal{S}|^2+2\alpha|\tau_g\phi|^2-\alpha'(t)|\nabla\phi|^2,$ and
\begin{align*}
2\mathcal{S}(\nabla{h},\nabla{h})&=2\text{Rc}(\nabla{h},\nabla{h})-2\alpha\nabla\phi\otimes \nabla\phi(\nabla{h},\nabla{h})=2\text{Rc}(\nabla{h},\nabla{h})-2\alpha\left\langle{\nabla{\phi},\nabla{h}}\right\rangle^2\\
\triangle{|\nabla{h}|^2}&= 2\Hs(h)^2+2\left\langle{\nabla{h},\nabla{\triangle{h}}}\right\rangle+2\text{Rc}(\nabla{h},\nabla{h}),
\end{align*}
where the second equation is by Bochner's identity. 

Combining those above yields
\begin{align*}
(\partial_{t}+\triangle)q=&4\left\langle{\mathcal{S},\text{Hess}(h)}\right\rangle+\triangle|\nabla{h}|^2-2\mathcal{S}(\nabla{h},\nabla{h})\\
&-2\left\langle{\nabla{h},\nabla(-\triangle{h}+|\nabla{h}|^2-S)}\right\rangle+2|\mathcal{S}|^2+2\alpha|\tau_g\phi|^2-\alpha'(t)|\nabla\phi|^2\\
=&4\left\langle{\mathcal{S},\text{Hess}(h)}\right\rangle+2\left\langle{\nabla{h},\nabla q}\right\rangle+2\text{Hess}(h)^2\\
&+2|\mathcal{S}|^2+2|\tau_g\phi|^2+2\alpha\left\langle{\nabla{\phi},\nabla{h}}\right\rangle^2-\alpha'(t)|\nabla\phi|^2\\
=&2|\mathcal{S}+\text{Hess}(h)|^2+2\alpha\left(|\tau_g\phi|^2+\left\langle{\nabla{\phi},\nabla{h}}\right\rangle^2\right)+2\left\langle{\nabla{h},\nabla q}\right\rangle-\alpha'(t)|\nabla\phi|^2.
\end{align*}
Thus, 
\begin{align*}
H^{-1}\Box^{\ast}v&=q+S-|\nabla{h}|^2-\frac{n}{2\tau}+2|\nabla{h}|^2\\
&-2\tau\left(|\mathcal{S}+\text{Hess}(h)|^2+\alpha\left(|\tau_g\phi|^2+\left\langle{\nabla{\phi},\nabla{h}}\right\rangle^2\right)\right)\\
&=-2\tau\left(\left|\mathcal{S}+\text{Hess}(h)-\frac{g}{2\tau}\right|^2+\alpha\left(|\tau_g\phi|^2+\left\langle{\nabla{\phi},\nabla{h}}\right\rangle^2\right)-\frac{1}{2}\alpha'(t)|\nabla\phi|^2\right).
\end{align*}
The result follows by the positivity of $\alpha(t)$ and the fact that it is non-increasing.\\
\end{proof}

The only remaining ingredient needed for the proof is the following proposition: 
\begin{proposition} 
\label{limit0}
Let $v=\Big((T-t)(2\triangle{h}-|\nabla{h}|^2+S)+h-n\Big)H$. For $\Phi$ being a smooth positive solution to the heat equation, if  $\rho_{\Phi}(t)=\int_{M}v\Phi d\mu_{M}$, then $\lim_{t\rightarrow T}\rho_{\Phi}(t)=0$.
\end{proposition}

\begin{proof}

IBP yields
\begin{align*}
\rho_{\Phi}(t) &=\int_{M}\Big(\tau(2\triangle{h}-|\nabla h|^2+S)+h-n\Big)H\Phi d\mu_{M}\\
&=-\int_{M}2\tau\nabla{h}\nabla({H\Phi})d\mu_{M}-\int_{M}\tau|\nabla h|^2 H\Phi d\mu_{M}+\int_{M}(\tau S+h-n)H\Phi d\mu_{M}\\
&= \int_{M}\tau|\nabla h|^2 H\Phi d\mu_{M}-2\tau\int_{M} \nabla{\Phi}\nabla{h}H d\mu_{M}+\int_{M}(\tau S+h-n)H\Phi d\mu_{M}\\
&= \int_{M}\tau|\nabla h|^2 H\Phi d\mu_{M}-2\tau\int_{M}H\triangle{\Phi}d\mu_{M}+\int_{M}(\tau S+h-n)H\Phi d\mu_{M}\\
&= \int_{M}\tau|\nabla h|^2 H\Phi d\mu_{M}+\int_{M}hH\Phi d\mu_{M}-2\tau\int_{M}H\triangle{\Phi}d\mu_{M}+\int_{M}(\tau S-n)H\Phi d\mu_{M}.
\end{align*}
For the first term, one can use Lemmas \ref{conjestimateS} and \ref{gradconjS} on $M\times[\frac{\tau}{2},\tau]$ to find that
\begin{align*}
\tau\int_{M}|\nabla h|^2H\Phi d\mu_{M}&\leq (2+C_{1}\tau)\int_{M}\left(\ln{\left(\frac{C_{3}e^{-D\tau/3}}{H(4\pi\tau)^{n/2}}\right)}+C_{2}\tau\right)H\Phi d\mu_{M}\\
& \leq (2+C_{1}\tau)\int_{M}\left(\ln{C_{3}}-\frac{D\tau}{3}+h+C_{2}\tau\right)H\Phi d\mu_{M},
\end{align*}
with $C_{1},C_{2}$ as in Lemma \ref{gradconjS} while $C_{3}=\frac{e^{B}}{2^{n/2}}$.\\
By applying Lemma \ref{integralboundabove}, 
\begin{align*}
\lim_{\tau\rightarrow 0}\left(\int_{M}\tau|\nabla h|^2d\mu_{M}+\int_{M}hH\Phi d\mu_{M}\right)&\leq 3\int_{M}hH\Phi d\mu_{M}+2\ln{C_{3}}\Phi(x,T)\\
&\leq \left(\frac{3n}{2}+2\ln{C_{3}}\right)\Phi(x,T). 
\end{align*}
Observe that, except for the first two terms, all of them approach $-n\Phi(y,T)$ as $\tau\rightarrow 0$. Therefore
\begin{equation*}
\lim_{t\rightarrow T}\rho_{\Phi}(t)\leq C_{4}\Phi(x,T).
\end{equation*}
Furthermore, as $\Phi$ is a positive smooth function satisfying the heat equation $\partial_{t}\Phi=\triangle{\Phi}$, one obtains that
\begin{equation}
\label{evolintegral}
\partial_{t}\rho_{\Phi}(t)=\partial_{t}\int_{M}v\Phi d\mu_{M}=\int_{M}(\Box{\Phi}v-\Phi\Box^{\ast}v)d\mu_{M}\geq 0.
\end{equation}
The above conditions imply that there exists $\beta$ such that 
\begin{equation*}
\lim_{t\rightarrow T}\rho_{\Phi}(t)=\beta.
\end{equation*}
Hence $\lim_{\tau\rightarrow 0}(\rho_{\Phi}(T-\tau)-\rho_{\Phi}(T-\frac{\tau}{2}))=0$. By equation (\ref{evolintegral}),  and the mean-value theorem, there exists a sequence $\tau_{i}\rightarrow 0$ such that 
\begin{equation*}
\lim_{\tau_{i}\rightarrow 0}\tau_{i}^2\int_{M}\left(\left|\mathcal{S}+\text{Hess}h-\frac{g}{2\tau}\right|^2+\alpha\left(|\tau_g\phi|^2+\left\langle{\nabla{\phi},\nabla{h}}\right\rangle^2\right)-\frac{1}{2}\alpha'(t)|\nabla\phi|^2\right)H\Phi d\mu_{M}=0.
\end{equation*}
Using standard inequalitites yield,
\begin{align*}
&\left(\int_{M}\tau_{i}\left(S+\triangle{h}-\frac{n}{2\tau_{i}}\right)H\Phi d\mu_{M}\right)^2 \\
&\leq \left(\int_{M}\tau_{i}^2\left(S+\triangle{h}-\frac{n}{2\tau_{i}}\right)^2H\Phi d\mu_{M}\right)\left(\int_{M}H\Phi d\mu_{M}\right)\\
& \leq \left(\int_{M}\tau_{i}^2\left|\mathcal{S}+\text{Hess}h-\frac{g}{2\tau}\right|^2H\Phi d\mu_{M}\right)\left(\int_{M}H\Phi d\mu_{M}\right).
\end{align*} 
Since $\lim_{\tau_{i}\rightarrow 0}\int_{M}H\Phi d\mu_{M}=\Phi(y,T)<\infty$ and $\alpha\left(|\tau_g\phi|^2+\left\langle{\nabla{\phi},\nabla{h}}\right\rangle^2\right)-\frac{1}{2}\alpha'(t)|\nabla\phi|^2\geq 0$, 
\begin{equation*}
\lim_{\tau_{i}\rightarrow 0}\int_{M}\tau_{i}\left(S+\triangle{h}-\frac{n}{2\tau_{i}}\right)H\Phi d\mu_{M}=0.
\end{equation*}
Therefore, by Lemma \ref {integralboundabove},
\begin{align*}
\lim_{t\rightarrow T}\rho_{\Phi}(t)&=\lim_{t\rightarrow T}\int_{M}(\tau_{i}(2\triangle{h}-|\nabla{h}|^2+S)+h-n)H\Phi d\mu_{M}\\
&=\lim_{t\rightarrow T}\int_{M}\left(\tau_{i}(\triangle{h}-|\nabla{h}|^2)+h-\frac{n}{2}\right)H\Phi d\mu_{M}\\
&=\lim_{t\rightarrow T}\left(\int_{M}-\tau_{i}H\triangle{\Phi}d\mu_{M}+\int_{M}\left(h-\frac{n}{2}\right)H\Phi d\mu_{M}\right)\\
&=\int_{M}\left(h-\frac{n}{2}\right)H\Phi d\mu_{M}\leq 0. 
\end{align*}
Hence $\beta\leq 0$. To show that equality holds, we proceed by contradiction. Without loss of generality, we may assume $\Phi(y,T)=1$. Let $H\Phi=(4\pi\tau)^{-n/2}e^{\tilde{h}}$ (that is, $\tilde{h}=h-\ln{\Phi}$), then IBP yields,
\begin{equation}
\rho_{\Phi}(t)=\mathbb{W}_\alpha(g,u,\tau,\tilde{h})+\int_{M}\left(\tau\left(\frac{|\nabla{\Phi}|^2}{\Phi}\right)-\Phi\ln{\Phi}\right)H d\mu_{M}.
\end{equation}
By the choice of $\Phi$ the last term converges to 0 as $\tau\rightarrow 0$. So if $\lim_{t\rightarrow T}\rho_{\Phi}(t)=\beta<0$ then $\lim_{\tau\rightarrow 0}\mu_{w}(g,u,\tau)<0$ and, thus, contradicts Lemma \ref{basicmu}. Therefore the only possibility is that $\beta=0$.
\end{proof}
\subsection{Proof of theorem \ref{thm-Harnack}}
\begin{proof}
 Recall from inequality (\ref{evolintegral})  
$$\partial_{t}\int_{M}v\Phi d\mu_{M}=\int_{M}(\Box{h}v-h\Box^{\ast}v)d\mu_{M}\geq 0.$$
By Proposition \ref{limit0}, $\lim_{t\rightarrow T}\int_{M}v\Phi d\mu_{M}=0$. Since $\Phi$ is arbitrary, $v\leq 0$.
\end{proof}

\subsection{Proof of corollary \ref{cor-LYH}}

This follows from standard arguments, due to Perelman \cite{perelman1}.  

\section{Sobolev imbedding theorems}\label{doi}

Now we turn our attention to a different approach, that of bounding the heat kernel by means of a Sobolev imbedding theorem under the $(RH)_\alpha$-flow. We first present the Sobolev inequalities that form the basis of our exploration. 

Sharpening a result by T. Aubin (\cite{TA76}), E. Hebey proved in (\cite{EHMV96}) the following : 
\begin{theorem}\label{thm_Aubin}
Let $M^n$ be a smooth compact Riemannian manifold of dimension $n$. Then there exists a constant $B$ such that for any $\psi\in W^{1,2}(M)$ (the Sobolev space of weakly differentiable functions) :
\[||\psi||_p^2\leq K(n,2)^2||\nabla\psi||_2^2+B||\psi||_{2}^2 \hspace{0.5cm}.\]

Here $K(n,2)$ is the best constant in the Sobolev imbedding (inequality) in $\mathbb{R}^n$ and $p=(2n)/(n-2)$. $B$ depends on the lower bound of the Ricci curvature and the derivatives of the curvature tensor. 
\end{theorem}

Note that Hebey's result was shown for complete manifolds, and in that situation $B$ depends on the injectivity radius. However, we are interested in compact manifolds, so $B$ will not depend on the injectivity radius. 

Later, along the Ricci flow, Q. Zhang proved the following uniform Sobolev inequality in \cite{bookCRC2010}:

\begin{theorem}\label{thm_Zhang}
Let $M^n$ be a compact Riemannian manifold, with $n\geq 3$ and let $\big(M,g(t)\big)_{t\in[0,T]}$ be a solution to the Ricci flow $\frac{\partial g}{\partial t}=-2\Ric$. Let $A$ and $B$ be positive numbers such that for $(M,g(0))$ the following Sobolev inequality holds: for any $v\in W^{1,2}(M,g(0))$,
\[\left(\int_M|v|^{\frac{2n}{n-2}}\ d\mu(g(0))\right)^{\frac{n-2}{n}}\leq A\int_M|\nabla v|^2\ d\mu(g(0))+B\int_Mv^2\ d\mu(g(0))\]

Then there exist positive functions $A(t)$, $B(t)$ depending only on the initial metric $g(0)$ in terms of $A$ and $B$, and $t$ such that, for all $v\in W^{1,2}(M,g(t))$, $t>0$, the following holds
 
\[\left(\int_M|v|^{\frac{2n}{n-2}}\ d\mu(g(t))\right)^{\frac{n-2}{n}}\leq A(t)\int_M\left(|\nabla v|^2+\frac{1}{4}Rv^2\right)\ d\mu(g(t))+B(t)\int_Mv^2\ d\mu(g(t))\]

Here $R$ is the scalar curvature with respect to $g(t)$. Moreover, if $R(x,0)>0$, then $A(t)$ is independent of $t$ and $B(t)=0$.
\end{theorem}

The proof of this theorem relies on the analysis of $\lambda_0$, which is the first eigenvalue of Perelman's $\mathcal{F}$-entropy, i.e.:
\[\lambda_0=\inf\limits_{||v||_2=1}\int_M(4|\nabla v|^2+Rv^2)\ d\mu(g(0)).\]

Recently it has been proven that the same theorem holds for the Ricci flow coupled with the harmonic map flow (see \cite{LXGWK13}), where the analysis is now based on 
\[\lambda_0^\alpha=\inf\limits_{||v||_2=1}\int_M(4|\nabla v|^2+Sv^2)\ d\mu(g(0)),\]
where $S=R-\alpha|\nabla\phi|^2$. In the new setting, at time $t$ there are positive functions $A(t)$ and $B(t)$ such that   
\[\left(\int_M|v|^{\frac{2n}{n-2}}\ d\mu(g(t))\right)^{\frac{n-2}{n}}\leq A(t)\int_M\left(|\nabla v|^2+\frac{1}{4}Sv^2\right)\ d\mu(g(t))+B(t)\int_Mv^2\ d\mu(g(t))\]

Moreover, if $\lambda_0^\alpha>0$, which is automatically satisfied if at the initial time $R(0)>\alpha (0)|\nabla\phi(0)|$, then $A(t)$ is a constant and $B(t)=0$.  

Recall that R. M\"uller introduced in \cite{Muller09} the $\mathcal{F}_{\alpha}$ and the $\mathbb{W}_\alpha$ functionals, which are the natural analogues of the $\mathcal{F}$ and the $\mathcal{W}$ functionals for the Ricci flow, introduced by Perelman.   

\section{Proof of theorem \ref{theorem} and its corollary}\label{patru}


We start the proof by assuming, without loss of generality, that $s=0$. By the semigroup property of the heat kernel \cite[Theorem 2.6]{guenther02} and the Cauchy-Bunyakovsky-Schwarz inequality we have that:
\begin{align*}
H(x,0;y,t)& =\int_M H\left(x,0;z,\frac{t}{2}\right)H\left(z,\frac{t}{2};y,t\right)d\mu\left(z,\frac{t}{2}\right) \\
          & \leq \left[\int_M H^2\left(x,0;z,\frac{t}{2}\right)d\mu\left(z,\frac{t}{2}\right) \right]^{1/2}\left[\int_M H^2\left(z,\frac{t}{2};y,t\right)d\mu\left(z,\frac{t}{2}
\right)\right]^{1/2}
\end{align*}

The key of the proof consists in determining upper bounds for the following two quantities:
\begin{align*}
\alpha(t)=\int_M H^2(x,s;y,t)d\mu(y,t)& \text{ (for $s$ fixed)}\\
\beta(s)=\int_M H^2(x,s;y,t)d\mu(x,s)& \text{ (for $t$ fixed)}
\end{align*}

We will find an ordinary differential inequality for each of the two.

We first deduce a bound on $\alpha(t)$, by finding an inequality involving $\alpha'(t)$ and $\alpha(t)$. Note that we will treat $H$ as being a function of $(x,t)$, the $(y,s)$ part is fixed. 

Since $\frac{d}{dt}(d\mu)=-Sd\mu=(-R+\alpha|\nabla\phi|^2)d\mu$, one has:
\begin{align}\label{pprim}
\alpha'(t) & =2\int_M H\cdot H_t\ d\mu(y,t)-\int_M H^2 (R-\alpha|\nabla\phi|^2)\ d\mu(y,t) \notag\\
 & = 2\int_M H\cdot(\triangle H)\ d\mu(y,t)-\int_M H^2 (R-\alpha|\nabla\phi|^2)\ d\mu(y,t)\notag \\
      &= -2\int_M|\nabla H|^2\ d\mu-\int_M H^2(R-\alpha|\nabla\phi|^2)\ d\mu \notag \\ 
      & \leq -\int_M[|\nabla H|^2+(R-\alpha|\nabla\phi|^2) H^2]\ d\mu(y,t)
\end{align}

Estimating $\int_M|\nabla H|^2d\mu$ will make use of the Sobolev imbedding theorem, which gives a relation between $\int_M|\nabla H|^2d\mu$ and $\int_M H^2d\mu$, and the H\"older inequality to bound the term involving $H^{2n/(n-2)}$:
\begin{align}\label{Holder}
\int_M H^2\ d\mu(y,t)\leq \left[\int_M H^{\frac{2n}{n-2}}\ d\mu(y,t) \right]^{\frac{n-2}{n+2}} \left[\int_M H\ d\mu(y,t) \right]^{\frac{4}{n+2}}
\end{align}

By theorem (\ref{thm_Aubin}) one gets that at time $t=0$, the following inequality holds for any $v\in W^{1,2}(M,g(0))$ (hence also for $H(x,s;y,t)$, which is smooth) and for some $B>0$:
\[\left(\int_M|v|^{\frac{2n}{n-2}}\ d\mu(g(0))\right)^{\frac{n-2}{n}}\leq K(n,2)^2\int_M|\nabla v|^2\ d\mu(g(0))+B\int_Mv^2\ d\mu(g(0))\]

Then by Theorem \ref{thm_Zhang} (applied for the $(RH)_\alpha$ flow) it follows that at any time $t\in (0,T]$ and for all $v\in W^{1,2}(M,g(t))$:
\[\left(\int_M|v|^{\frac{2n}{n-2}}\ d\mu(g(t))\right)^{\frac{n-2}{n}}\leq A(t)\int_M\left(|\nabla v|^2+\frac{1}{4}(R-\alpha|\nabla\phi|^2)v^2\right)\ d\mu(g(t))+B(t)\int_Mv^2\ d\mu(g(t))\]

where $A(t)$ is a positive function depending on $g(0)$ and $K(n,2)^2$, while  $B(t)$ is also a positive function, depending on $B$, which in turn depends on the initial Ricci curvature on $M$ and on the derivatives of the curvatures on $M$ at time $0$.

Applying the above for the heat kernel, one can relate the RHS of (\ref{Holder}) to the Sobolev inequality:
\begin{gather}\label{Sobolev1}
\int_M H^2\ d\mu(y,t)  \leq \left[\int_M H^{\frac{2n}{n-2}}\ d\mu(y,t) \right]^{\frac{n-2}{n+2}} \left[\int_M H\ d\mu(y,t) \right]^{\frac{4}{n+2}} \notag\\
\leq \left[ A(t)\int_M\left(|\nabla H|^2+\frac{1}{4}(R-\alpha|\nabla\phi|^2)H^2\right)\ d\mu(y,t)+B(t)\int_MH^2\ d\mu(y,t)\right]^{\frac{n}{n+2}}\left[\int_M H\ d\mu(y,t)\right]^{\frac{4}{n+2}}
\end{gather}

We will now focus our attention to $J(t):=\int_M H(x,s;y,t)\ d\mu(y,t)$. By the definition of the fundamental solution $\int_M H(x,s;y,t)d\mu(x,s)=1$, but that's not true if one integrates in $(y,t)$. Our goal will be to obtain a differential inequality for $J(t)$, from which a bound will be found. 
\begin{align*}
J'(t)&=\int_M H_t(x,t;y,s)\ d\mu(y,t)+\int_M H(x,s;y,t)\frac{d}{dt}\ d\mu(y,t)=\int_M\triangle_yH(x,s;y,t)\ d\mu(y,t)\\
     &-\int_M H(x,s;y,t)S(y,t)\ d\mu(y,t)=-\int_M H(x,s;y,t)S(y,t)\ d\mu(y,t)
\end{align*}
the first term being $0$, as $M$ is a compact manifold, without boundary.

Recall that $S$ satisfies the following equation (Theorem 4.4 from \cite{Muller09}):
\begin{align*}
\frac{\partial S}{\partial t}=\bigtriangleup S+2|S_{ij}|^2+2\alpha|\tau_g\phi|^2-\frac{\partial\alpha}{\partial t}|\nabla\phi|^2
\end{align*}
But $|S_{ij}|^2\geq \frac{1}{n}S^2$ (this is true for any 2-tensor) and $\alpha(t)$ is a positive function, non-increasing in time, so one obtains: 
\[\frac{\partial S}{\partial t}-\bigtriangleup S-\frac{2}{n}S^2\geq 0\]

Since the solutions of the ODE $\frac{d\rho}{dt}=\frac{2}{n}\rho^2$
are $\rho(t)=\frac{n}{n\rho(0)^{-1}-2t}$, by the maximum principle
we get a bound on $S$, for $s\leq \tau\leq t$:
\begin{align*}
S(z,\tau)\geq \frac{n}{n(\inf_{t=0}S)^{-1}-2\tau}=\frac{1}{(\inf_{t=0}S)^{-1}-\frac{2}{n}\tau}:=\frac{1}{m_0-c_n\tau}
\end{align*}
(here and later, if $\inf_{t=0} S\geq 0$, then the above is regarded
as zero).

\bigskip
Using this lower bound for $S$ (for $\tau\in(s,t]$), we get:
\begin{align*}
J'(\tau)\leq -\frac{1}{m_0-c_n\tau}J(\tau)
\end{align*}
After integrating the above from $s$ to $t$, while noting that by $J(s)$ one understands: 
\[J(s)=\lim\limits_{t\to s}\int_M H(x,s;y,t)\ d\mu(y,t)=\int_M \lim\limits_{t\to s} H(x,s;y,t)\ d\mu(y,t)=\int_M \delta_{y}(x)\ d\mu(x,s)=1\]
one obtains:
\begin{align*}
J(t)\leq \left(\frac{m_0-c_nt}{m_0-c_ns}\right)^{\frac{n}{2}}:=(\chi_{t,s})^{\frac{n}{2}}
\end{align*}

Hence $\int_M H(x,s;y,t)\ d\mu(y,t)\leq (\chi_{t,s}))^{\frac{n}{2}}$ and (\ref{Sobolev1}) becomes:

\begin{align*}
\int_M H^2d\mu(y,t)\leq \left[ A(t)\int_M\left(|\nabla H|^2+\frac{1}{4}SH^2\right)\ d\mu(y,t)+B(t)\int_MH^2\ d\mu(y,t)\right]^{\frac{n}{n+2}}\left(\chi_{t,s}\right)^{\frac{2n}{n+2}}
\end{align*}

From this it follows immediately that:
\begin{align*}
\int_M |\nabla H|^2\ d\mu(y,t)\geq \frac{1}{\chi^2_{t,s}A(t)}\left[\int_M H^2\ d\mu(y,t)\right]^{\frac{n+2}{n}} -\frac{B(t)}{A(t)}\int_M H^2\ d\mu(y,t)-\frac{1}{4}\int SH^2\ d\mu(y,t)
\end{align*}

Combining this with the inequality from (\ref{pprim}), one obtains the following differential inequality for $\alpha(t)$:
\begin{align*}
\alpha'(t)\leq -\frac{1}{\chi^2_{t,s}A(t)} \alpha(t)^{\frac{n+2}{n}}+\frac{B(t)}{A(t)}\alpha(t)-\frac{3}{4}\int SH^2d\mu(y,t)
\end{align*}

Note that the above is true for any $\tau\in(s,t]$. For the following computation, we will consider $t$ fixed as well.  Recall that for $\tau\in (s,t]$, $S(\cdot, \tau)\geq\frac{1}{m_0-c_n\tau}$. Denoting with:
\begin{align*}
f(\tau):=\frac{B(\tau)}{A(\tau)}-\frac{3}{4}\cdot\frac{1}{m_0-c_n\tau}
\end{align*}
we get:
\begin{align*}
\alpha'(\tau)\leq -\frac{1}{\chi^{2}_{\tau,s}A(\tau)} \alpha(\tau)^{\frac{n+2}{n}}+f(\tau)\alpha(\tau)
\end{align*}

Let $F(\tau)$ be an antiderivative of $h(\tau)$. By the integrating factor method, one finds:
\begin{align*}
(e^{-F(\tau)}\alpha(\tau))'\leq  -\frac{1}{\chi^{2}(\tau)A(\tau)}(e^{-F(\tau)}\alpha(\tau))^{\frac{n+2}{n}}e^{\frac{2}{n}F(\tau)}
\end{align*}
Since the above is true for any $\tau\in(s,t]$, by integrating from $s$ to $t$ and taking into account that 
\[\lim\limits_{\tau\searrow s}\alpha(\tau)=\int_M\lim\limits_{\tau\searrow s}H^2(x,\tau;y,s)\ d\mu(x,\tau)=\int_M\delta^2_y(x)\ d\mu(x,s)=0\]
one obtains the first necessary bound:
\begin{align*}
\alpha(t)\leq \frac{C_ne^{F(t)}}{\left(\int\limits_{s}^{t}\frac{e^{\frac{2}{n}F(\tau)}}{\chi^{2}(\tau)A(\tau)} d\tau\right)^{\frac{n}{2}}}
\end{align*}
where $C_n=\left(\frac{2}{n}\right)^{\frac{n}{2}}$.

\bigskip
The next step is to estimate $\beta(s)=\int_M H^2(x,t;y,s)\ d\mu(x,s)$, for which the computation is different, due to the asymmetry of the
equation. As stated above, the second entries of $H$ satisfy the conjugated equation:
\begin{align*}
\triangle_x H(x,s;y,t)+\partial_s H(x,s;y,t)- S H(x,s;y,t)=0
\end{align*}

Proceeding just as in the $\alpha(t)$ case, we get the following:
\begin{align*}
\beta'(s)& =2\int_M HH_s\ d\mu(x,s)-\int_M SH^2\ d\mu(x,s)\\
         &=2\int_M H(-\triangle H +SH)\ d\mu(x,s)-\int_M SH^2\ d\mu(x,s)\\
       & =-2\int_M H(\triangle H)\ d\mu(x,s)+\int_M SH^2\ d\mu(x,s)\\
       & =2\int_M |\nabla H|^2\ d\mu(x,s)+\int_M SH^2\ d\mu(x,s) \\
     & \geq \int_M |\nabla H|^2\ d\mu(x,s)+\int_M SH^2\ d\mu(x,s)
\end{align*}

Hence
\begin{align*}
\beta'(s)\geq \int_M (|\nabla H|^2+SH^2)\ d\mu(x,s)
\end{align*}

But this time, by the property of the heat kernel:
\begin{align*}
\tilde{J}(s):=\int_M H(x,s;y,t)\ d\mu(x,s)=1
\end{align*}

so by applying H\"older (as for $\alpha(t)$) and relating it to the Sobolev inequality, we get:
\begin{align*}
\int_M H^2\ d\mu(x,s)\leq \hspace{6cm}  \\
 \left[ A(s)\int_M\left(|\nabla H|^2+\frac{1}{4}SH^2\right)\ d\mu(x,s)+B(s)\int_MH^2\ d\mu(x,s)\right]^{\frac{n}{n+2}}\left[\int_M H\ d\mu(x,s)\right]^{\frac{4}{n+2}} \\
 =\left[ A(s)\int_M\left(|\nabla H|^2+\frac{1}{4}SH^2\right)\ d\mu(x,s)+B(s)\int_MH^2\ d\mu(x,s)\right]^{\frac{n}{n+2}}
\end{align*}

Following the same steps as for $\alpha(t)$, one finds
\begin{align*}
\beta'(s)\geq \frac{1}{A(s)} \beta(s)^{\frac{n+2}{n}}-f(s)\beta(s)
\end{align*}
($f(s)$ denotes, as before, $\frac{B(s)}{A(s)}-\frac{3}{4}\cdot\frac{1}{m_0-c_ns}$)

The above is true for any $\tau\in[s,t)$. We will apply again the integrating factor method, with $F(\tau)$ being the same antiderivative of $f(\tau)$ as above. 
For $\tau\in [s,t)$, the following holds:
\begin{align*}
(e^{F(\tau)}\beta(\tau))'\geq \frac{1}{A(\tau)}(e^{F(\tau)}\beta(\tau))^{\frac{n+2}{n}}e^{-\frac{2}{n}F(\tau)}
\end{align*}

Integrating between $s$ and $t$, and taking into account that
\[ \lim\limits_{\tau\nearrow t}\beta(\tau)=\int_M\lim\limits_{\tau\nearrow t}H^2(x,t;y,\tau)\ d\mu(y,\tau)=\int_M\delta^2_y(x)\ d\mu(y,t)=0\] 
we get the second desired bound:
\begin{align*}
\beta(s)\leq \frac{C_ne^{-F(s)}}{\left(\int\limits_{s}^{t} \frac{e^{-\frac{2}{n}F(\tau)}}{A(\tau)}\ d\tau\right)^{n/2}}
\end{align*}

From the estimates of $\alpha$ and $\beta$ we get the following:

\begin{align*}
\alpha\left(\frac{t}{2}\right)=\int_M H^2\left(x,0;z,\frac{t}{2}\right)\ d\mu\left(z,t/2\right)\leq \frac{C_ne^{F(t/2)}}{\left( \int\limits_{0}^{t/2}\left(\frac{m_0-c_n\tau}{m_0}\right)^{-2} \frac{e^{\frac{2}{n}F(\tau)}}{A(\tau)} \ d\tau\right)^{\frac{n}{2}}}
\end{align*}
\begin{align*}
\beta\left(\frac{t}{2}\right)=\int_M H^2\left(z,\frac{t}{2};y,t\right)\ d\mu\left(z,\frac{t}{2}\right)\leq \frac{C_ne^{-F(t/2)}}{\left(\int\limits_{t/2}^{t}
\frac{e^{-\frac{2}{n}F(\tau)}}{A(\tau)} \ d\tau\right)^{n/2}}
\end{align*}

Here, we may choose $F(t/2)=\int\limits_{0}^{t/2}\left[\frac{B(\tau)}{A(\tau)}-\frac{3}{4}\cdot\frac{1}{m_0-c_n\tau}\right]\ d\tau$, since the relation is true for any antiderivative of $f(\tau)=\frac{B(\tau)}{A(\tau)}-\frac{3}{4}\cdot\frac{1}{m_0-c_n\tau}$.

The conclusion follows from multiplying the relations above.


\subsection{Proof of the corollary}

In the special case when $S(x,0)>0$, then $S(x,t)>0$ for all $t>0$, so it follows that $J'(\tau)\leq 0$. $J(\tau)$ is thus decreasing, so $J(\tau)\leq J(s)=1$, which leads to the differential inequality for $\alpha(t)$ to be: 

\begin{align*}
\alpha'(t)\leq -\frac{1}{A(t)} \alpha(t)^{\frac{n+2}{n}}+\frac{B(t)}{A(t)}\alpha(t)
\end{align*}

And from this the bound for $\alpha(t)$ becomes:
\begin{align*}
\alpha(t)\leq \frac{C_ne^{F(t)}}{\left(\int\limits_{s}^{t}\frac{e^{\frac{2}{n}F(\tau)}}{A(\tau)} d\tau\right)^{\frac{n}{2}}}
\end{align*}
where $F(\tau)$ is the anti-derivative of $\frac{B(\tau)}{A(\tau)}$ such that $F(s)\neq 0$ and $F(t)\neq 0$.

Similarly, one obtains for $\beta(s)$:
\begin{align*}
\beta'(s)\geq \frac{1}{A(s)} \beta(s)^{\frac{n+2}{n}}-\frac{B(s)}{A(s)}\alpha(s)
\end{align*}
and from this: 
\begin{align*}
\beta(s)\leq \frac{C_ne^{-F(s)}}{\left(\int\limits_{s}^{t} \frac{e^{-\frac{2}{n}F(\tau)}}{A(\tau)}\ d\tau\right)^{n/2}}
\end{align*}
where $F(\tau)$ is the same anti-derivative of $\frac{B(\tau)}{A(\tau)}$ as above. 

By (\ref{thm_Zhang}), in the case of $S(x,0)>0$, the function $A(t)$ is a constant, while $B(t)=0$. Recall that $A(t)=A(0)$ is in fact $K(n,2)$, where $K(n,2)$ is the best constant in the Sobolev imbedding.

One has then that $F(t)=\frac{B}{A}t=0$. Using this, we get:
\begin{align*}
H(x,s;y,t)\leq &
\frac{C_n}{\left(\int\limits_{s}^{\frac{s+t}{2}}\frac{1}{A(0)} \ d\tau\right)^{\frac{n}{4}} \left(\int\limits_{\frac{s+t}{2}}^{t}
\frac{1}{A(0)} \ d\tau\right)^{\frac{n}{4}}}=\frac{C_n}{\left[\left(\frac{t-s}{2A}\right)^2\right]^{\frac{n}{4}}} = \frac{\tilde{C}_n}{(t-s)^{\frac{n}{2}}}
\end{align*}

where $\tilde{C}_n=C_n\cdot (2A)^{\frac{n}{2}}=\left(\frac{4K(n,2)}{n}\right)^{\frac{n}{2}}$.

This proves the desired corollary.

\subsection{Proof of theorem \ref{comp-theorem}}

It is interesting to compare the two estimates on the heat kernel, the one appearing in Lemma \ref{conjestimateS} and the one in this last corollary. Assuming that $S(x,0)>0$, then Lemma \ref{conjestimateS} showed that
\[H(x,t;y,T) \leq e^{B-(T-t)D/3}(4\pi (T-t))^{-n/2}\] where $B =-\inf\limits_{0<\tau\leq T} \mu_{\alpha}(g,\phi,\tau)$ and $D =\inf_{M\times \{0\}}{S}$.  

However, by the corollary to theorem \ref{theorem}, one has that \[H(x,t;y,T)\leq \tilde{C}_n(T-t)^{-n/2}.\]

Since $\tilde{C}_n$ is a universal constant, one can conclude that 
\[B\leq \frac{(T-t)D}{3}\ln\left[(4\pi)^{n/2}\tilde{C}_n\right].\]

Therefore, one has the following inequality: 
\[\mu_\alpha(g,\phi,\tau)\geq \frac{\tau D}{3}\ln\left[(4\pi)^{n/2}\tilde{C}_n\right]\]
where  $\mu_\alpha$ is the associated functional $\mu_\alpha(g,\phi,\tau)=\inf\limits_f\mathbb{W}_\alpha(g,\phi,\tau, f)$,  $D=\inf_{M\times \{0\}}{S}$ and $\tilde{C}_n=\left(\frac{4K(n,2)}{n}\right)^{\frac{n}{2}}$.


\bibliographystyle{unsrt}
\bibliography{bio}

\end{document}